\documentclass[12pt]{amsart}
\usepackage[utf8]{inputenc}
\usepackage{amsrefs}
\usepackage{amsfonts}
\usepackage{amsmath}
\usepackage{amssymb}
\usepackage{amsthm}
\usepackage[margin=1.6in]{geometry}
\usepackage{graphicx}
\usepackage{hyperref}
\hypersetup{
  colorlinks=true,
  linkcolor=blue,
  citecolor=red
}

\newtheorem{theorem}{Theorem}
\newtheorem*{theorem*}{Theorem}
\newtheorem{lemma}[theorem]{Lemma}
\newtheorem{proposition}[theorem]{Proposition}
\newtheorem{claim}[theorem]{Claim}
\newtheorem{corollary}[theorem]{Corollary}

\newtheorem{conjecture}[theorem]{Conjecture}
\newtheorem{definition}{Definition}

\newcommand{\vs}{\vspace{3mm}}

\newcommand{\hs}{\hspace{1mm}}
\newcommand{\R}{\mathbb{R}}
\newcommand{\C}{\mathbb{C}}
\newcommand{\Z}{\mathbb{Z}}
\newcommand{\Q}{\mathbb{Q}}

\newcommand{\mc}{\mathcal}

\newcommand{\ep}{\epsilon}

\newcommand{\al}{\alpha}

\newcommand{\weakly}{\rightharpoonup}

\newcommand{\lan}{\langle}
\newcommand{\ran}{\rangle}

\def\cc{\curvearrowright}
\def\cH{\mathcal{H}}
\def\bH{b\mathcal{H}}
\def\R{\mathbb{R}}
\def\Z{\mathbb{Z}}
\def\T{\mathbb{T}}

\newcommand{\wade}[1]{{\color{green}{ Wade: #1}}}

\title{Additive Conjugacy and the Bohr Compactification of Orthogonal Representations}
\author{Zachary Chase,  Wade Hann-Caruthers, and Omer Tamuz}
\address{California Institute of Technology, Pasadena, CA, 91125}
\thanks{Omer Tamuz was supported by a grant from the Simons Foundation
  (\#419427), a Sloan research fellowship, a BSF award (\#2018397),
  and an NSF CAREER award (DMS-1944153).}
\begin{document}

\begin{abstract}
We say that two unitary or orthogonal representations of a finitely generated group $G$ are {\em additive conjugates} if they are intertwined by an additive map, which need not be continuous. We associate to each representation of $G$ a topological action that is a complete additive conjugacy invariant: the action of $G$ by group automorphisms on the Bohr compactification of the underlying Hilbert space. Using this construction we show that the property of having almost invariant vectors is an additive conjugacy invariant.  As an application we show that $G$ is amenable if and only if there is a nonzero homomorphism from $L^2(G)$ into $\R/\Z$ that is invariant to the $G$-action.
\end{abstract}

\maketitle

%\tableofcontents

%\newpage

\section{Introduction}

Let $G$ be a finitely generated group. To each unitary or orthogonal representation of $G$ one can associate a probability measure preserving action---the so-called Gaussian action. Conversely, to each probability measure preserving action of $G$ one can associate the Koopman representation.  These constructions have proven to be an important connection between ergodic theory and representation theory, with many interesting applications (see, e.g., \cite{connes1980property, schmidt1981amenability, schmidt1996infinitely, glasner1997kazhdan}). 

In this paper we associate a {\em topological} action to representations, with the goal of establishing connections between the dynamical properties of the representation and the action. We use this construction to study properties of the representation that are ``additive conjugacy'' invariants; we define this notion below. As an application we derive a new characterization of amenability.

Let $G \cc^\pi \cH$ be an orthogonal representation of a finitely generated group on a separable real Hilbert space.\footnote{Note that the class of orthogonal representations includes the unitary ones;  we elaborate on this in \S\ref{sec:representations}. Thus all of our results apply to unitary representations on separable complex Hilbert spaces.} 
We associate to the representation $\pi$ the topological action of $G$ on the {\em Bohr compactification} of $\cH$. In general, the Bohr compactification $bA$ of a topological abelian group $A$ is the algebraic dual of $\hat A_d$, where the latter is the algebraic dual of $A$, equipped with the discrete topology. As we explain in \S\ref{sec:duality}, in our case of a real separable Hilbert space $\cH$, an equivalent definition is to let $\bH$ be the set of all homomorphisms (i.e., additive maps) from $\cH$ to $\T = \R/\Z$:
\begin{align}
\label{eq:bohr}
  \bH = \{\varphi \colon \cH \to \T\,\vert\, \varphi(v+w) = \varphi(v)+\varphi(w)\}.
\end{align}
Importantly, $\bH$ includes maps that are not continuous. The Bohr compactification $\bH$ is indeed compact, when endowed with the subspace topology induced from the product space $\T^\cH$. It also admits the obvious abelian group structure, which is compatible with this topology.

The group $G$ acts on $\bH$ by precomposition:
\begin{align*}
    [g\varphi](v) = \varphi(\pi_g^{-1} v).
\end{align*}
It is straightforward to verify that this action is by automorphisms of $\bH$ as a topological group. Thus the action $G \cc \bH$ is a topological algebraic action that is associated to the representation $\pi$. This action will be useful in the study of the following notion of conjugacy: 
\begin{definition}
Two representations, $G \cc^{\pi} \cH$ and $G \cc^{\pi'} \cH'$, are {\em additive conjugates} if there exists a bijection $\xi \colon \cH \to \cH'$ such that for all $v,w \in \cH$ and $g \in G$,
\begin{align*}
    \xi(v+w) = \xi(v)+\xi(w)
\end{align*}
and
\begin{align*}
    \xi(\pi_g v) = \pi'_g \xi(v).
\end{align*}
\end{definition}
That is, two representations are additive conjugates if they are intertwined by  an additive bijection. Note that this bijection need not be continuous. 

It is straightforward to check that the action on the Bohr compactification is a complete additive conjugacy invariant. That is, that $\pi$ and $\pi'$ are additive conjugates if and only if $G \cc \bH$ and $G \cc \bH'$ are conjugates, as topological algebraic actions (Claim~\ref{clm:complete-invariant}).

\subsection{Main results}
In all of our results below, $G$ is a finitely generated group, Hilbert spaces are separable and either real or complex, and  representations are, respectively, either orthogonal or unitary---unless otherwise specified.

Our main result ties an important property of a representation  with a dynamical property of its associated  topological action. Recall that $\pi$ is said to have {\em almost invariant vectors} if there exists a sequence of unit vectors $(u_n)_n$ in $\cH$ such that $\lim_n \Vert \pi_g u_n-u_n\Vert=0$ for each $g \in G$. A {\em fixed point} $x$ of a topological action $G \cc X$ is one that satisfies $gx = x$ for all $g \in G$.
\begin{theorem}
\label{thm:fixed-point}
$G \cc^\pi \cH$ has almost invariant vectors if and only if the associated action $G \cc \bH$ has a nonzero fixed point.
\end{theorem}
Since the action on the Bohr compactification is a complete additive conjugacy invariant, this theorem has an immediate corollary.
\begin{corollary}
\label{cor:additive}
Let $\pi_1$ and $\pi_2$ be additive conjugates. Then $\pi_1$ has almost invariant vectors if and only if $\pi_2$ has almost invariant vectors.
\end{corollary}
We note that, as far as we know, this is not known even for the case of $G=\Z$. This corollary may be a priori surprising, since almost invariant vectors are defined using the topology of the Hilbert space, whereas this topology does not appear in the definition of additive conjugacy.

By the Hulanicki-Reiter Theorem (see, e.g., \cite[Theorem G.3.2]{bekka2008kazhdan}), $G$ is amenable if and only if the left regular real representation $G\cc^\lambda L^2(G)$ has almost invariant vectors. Hence the following is another corollary of Theorem~\ref{thm:fixed-point}:
\begin{corollary}
$G$ is amenable if and only if there exists a nonzero  homomorphism $\varphi \colon L^2(G) \to \T$ that is invariant to the $G$ action: $\varphi(f) = \varphi(\lambda_g f)$ for all $f \in L^2(G)$ and $g \in G$.
\end{corollary}
Note that this homomorphism is not necessarily continuous.

\subsection{Proof sketch}
Both directions of the proof of Theorem~\ref{thm:fixed-point} require some work. An important tool is the natural homomorphism $\sigma \colon \cH \to \bH$, which, for a real Hilbert space $\cH$, is given by
\begin{align*}
    [\sigma_v](w) = \lan v,w \ran + \Z.
\end{align*}

When $\pi$ has almost invariant vectors $(v_n)_n$, it is straightforward to show that any limit point of $(\sigma_{v_n})_n$ is a fixed point of $\bH$. However, this fixed point might be zero. To overcome this, we construct from $(v_n)_n$ a modified sequence of almost invariant vectors $(w_n)_n$ such that all limit points of $(\sigma_{w_n})_n$ are nonzero.

When $\pi$ does not have almost invariant vectors, we in fact prove a stronger statement. Given a symmetric probability measure $\mu$ on $G$ whose support is equal to a finite generating set of $G$ containing the identity, we say that $\varphi \in \bH$ is $\mu$-harmonic if 
\begin{align}
\label{eq:harmonic}
    \varphi\left(\sum_{g}\mu(g)\pi_g v\right) = \varphi(v)
\end{align}
for every $v \in \cH$. We prove the following claim, which implies the corresponding direction of Theorem~\ref{thm:fixed-point}.

\begin{proposition}
\label{prop:harmonic}
  Suppose $\pi$ does not have almost invariant vectors and $\mu$ is a symmetric generating measure for $G$. Then $\bH$ has no nonzero $\mu$-harmonic points, and in particular has no nonzero fixed points.
\end{proposition}

\subsection{Open questions and additional results}

This paper leaves unanswered the larger question of what properties of a representation are reflected in its Bohr compactification, or, equivalently, what properties are additive conjugacy invariants. 

In \S\ref{sec:Z-reps} we completely classify the additive conjugacy classes of the irreducible unitary representations of $\Z$, showing that these are determined by the Galois conjugacy class of the eigenvalue. In particular, the case of $\Z$ shows that  representations that are not conjugate (in the the usual sense) can be additive conjugates, and that furthermore this holds even for irreducible representations. A similar analysis should apply to $\Z^d$. For other groups, we leave this question for future research.

%It is impossible that additive conjugacy of representations is equivalent to (the usual notion of) conjugacy: a simple obstruction is provided by the example of the trivial representations on $\R$ and $\R^2$. These representations are additive conjugates since there exists an additive bijection between $\R$ and $\R^2$; but they are not conjugates, since such a map cannot be continuous. Nevertheless, we have no example of two irreducible (or even cyclic) representations that are not conjugates and yet are additive conjugates.

%This leaves two possibilities, both of which we find intriguing. The first is that---for irreducible representations---additive conjugacy coincides with conjugacy, and thus the action on the Bohr compactification somehow contains all of the data of the representation. The second is that these notions are not the same for irreducible representations, in which case it would be interesting to understand what properties of a representation are additive conjugacy invariants.

One may imagine that there is some connection between weak containment and additive conjugacy. Indeed, this is perhaps suggested by Corollary~\ref{cor:additive}. We prove an additional result in this direction. This result can be interpreted to imply that the Bohr compactification records the data of the weakly contained irreducible representations.

\begin{proposition}
\label{prop:weak-containment}
Let $G$ be a finitely generated group. Let $G \cc^\pi \cH$ be an orthogonal representation that weakly contains the irreducible orthogonal representation $G \cc^{\pi'} \cH'$. Then for every $v' \in \cH'$ there are $\varphi \in \bH$ and $v \in \cH$ such that for all $g \in G$,
\begin{align*}
    [g\varphi](v) = \lan \pi'_g  v', v'\ran+\Z.
\end{align*}
\end{proposition}

This proposition, together with Theorem~\ref{thm:fixed-point}, suggests the following (perhaps bold) conjecture.
\begin{conjecture}
\label{conj}
Let $G$ be a finitely generated group. Let $G \cc^\pi \cH$ be an orthogonal representation. Then the following are equivalent:
\begin{enumerate}
    \item $\pi$ weakly contains the irreducible orthogonal representation $G \cc^{\pi'} \cH'$. 
    \item   There is a closed, $G$-invariant subgroup $K \subseteq \bH$ such that the topological algebraic actions $G \cc K$ and $G \cc \bH'$ are conjugate.
\end{enumerate}
\end{conjecture}
%In \S\ref{sec:conjecture} we prove that this is indeed true when $\pi'$ is the trivial representation; this is a simple consequence of Theorem~\ref{thm:fixed-point}.

\subsubsection*{Acknowledgments}
We would like to thank Todor Tsankov for suggesting some improvements to our proofs, Yehuda Shalom for suggesting to us the  classification of the additive conjugacy classes of the irreducible representations of $\Z$, and Andreas Thom for pointing out an error in an earlier version, as well as suggesting a correct proof. We would also like to thank Joshua Frisch, Eli Glasner, Alexander Kechris, Jesse Peterson, Pooya Vahidi Ferdowsi, Benjamin Weiss, and Andy Zucker for helpful discussions.

\section{Definitions}

\subsection{Orthogonal and unitary representations}
\label{sec:representations}
Let $\cH$ be a separable real Hilbert space equipped with an inner product $\lan\cdot,\cdot\ran$. An {\em orthogonal representation} $\pi$ of a discrete group $G$ is a homomorphism $\pi \colon G \to O(\cH)$, where $O(\cH)$ is the group of orthogonal (i.e., linear and inner product preserving) bijections from $\cH$ to $\cH$. That is, $\pi$ is a group homomorphism of $G$ into the group of linear transformations of $\cH$ that preserve the inner product of $\cH$. We henceforth omit $\pi$ from our notation and write the image of $w \in \cH$ under $\pi_g$ simply as $g w$.

As the next lemma shows, every {\em unitary} representation on a complex Hilbert space $\cH$ is also an orthogonal representation of the associated real Hilbert space.
\begin{lemma}
\label{lemma:unitary}
Let $G \curvearrowright^\pi \cH$ be a unitary representation of $G$ on a complex Hilbert space $\cH$. Let $\widetilde{\cH}$ denote the realification of $\cH$, with inner product $\langle u,v \rangle_{\widetilde{\cH}} := \Re\langle u,v \rangle_\cH$. Let $\widetilde{\pi}: G \curvearrowright \widetilde{\cH}$ be the same action as $\pi$. Then $\widetilde{\pi}$ is an orthogonal representation, and $\pi$ has almost invariant vectors if and only if $\widetilde{\pi}$ does. 
\end{lemma}

\begin{proof}
For any $v,w \in \widetilde{\cH}$, $\langle gv, gw\rangle_{\widetilde{\cH}} = \Re\langle gv,gw\rangle_\cH = \Re\langle v,w\rangle_\cH = \langle v,w \rangle_{\widetilde{\cH}}$. The equivalence of having almost invariant vectors follows from the fact that the norms on $\cH$ and $\widetilde{\cH}$ are the same and that the actions are the same.
\end{proof}

It follows from Lemma~\ref{lemma:unitary} that it suffices to prove Theorem~\ref{thm:fixed-point} for orthogonal representations to conclude that it also holds for unitary representations. The same holds for Corollary~\ref{cor:additive} and Proposition~\ref{prop:harmonic}.

\subsection{Pontryagin duality and the Bohr compactification}
\label{sec:duality}
A homomorphism of a topological abelian group $A$ into $\T = \R/\Z$ is a map $\varphi$ that satisfies $\varphi(v+w)=\varphi(v)+\varphi(w)$ for all $v,w\in A$. The set of all {\em continuous} such homomorphisms, equipped with the compact-open topology, is called the algebraic dual of $A$ and is denoted by $\hat A$. The Bohr compactification $bA$ of $A$ is $\widehat{\hat A_d}$, where $\hat{A}_d$ is $\hat A$ equipped with the discrete topology. That is, $bA$ is  the set of {\em all} (i.e., not necessarily continuous) homomorphisms from $\hat A$ to $\T$.

The natural map $\tau \colon A \to bA$ given by
\begin{align*}
    [\tau(v)](\gamma) = \gamma(v)
\end{align*}
is known to be injective and continuous when $A$ is locally compact, in which case its image is dense in $bA$. 

Some groups, such as $\R$, can be (non-canonically) identified with their algebraic dual. In this case, $bA$ is simply the set of all homomorphisms from $A$ to $\T$. As we show (Proposition~\ref{prop:H-dual}) this identification holds for a separable real Hilbert space $\cH$. We hence define the Bohr compactification of $\cH$ as in \eqref{eq:bohr}, by letting $\bH$ be the algebraic dual of $\cH$  equipped with the discrete topology.

Since $\bH$ is compact, it follows from the Pontryagin Duality Theorem (see, e.g., \cite{hewitt2012abstract}) that its algebraic dual $\widehat \bH$ can be canonically identified with $\cH$ equipped with the discrete topology. This identification is realized by
\begin{align}
    \label{eq:duals}
    v(\varphi) = \varphi(v).
\end{align}

\subsection{Generating measures and harmonic homomorphisms}
Let $S$ be a finite, symmetric generating set for $G$ containing the identity, and let $\mu$ be a symmetric probability measure whose support is equal to $S$. We call such $\mu$ ``symmetric generating measures".

Let $P_\mu \colon \mathcal{H} \to \mathcal{H}$ be the continuous linear operator given by 
$$
  P_\mu w = \sum_{h \in S} \mu(h) hw,
$$ 
and let $D_u : \mathcal{H} \to \mathcal{H}$ be given by 
$$
  D_\mu w = w-P_\mu w.
$$ 
We say that $\varphi \in \bH$ is $\mu$-harmonic if $\varphi(D_\mu w)=0$ for all $w \in \cH$. By additivity, this is equivalent to the definition above, in \eqref{eq:harmonic}.

\section{General properties of the action on the Bohr compactification}

In this section we outline some simple, general properties of the action on the Bohr compactification and its relation to the representation. 

Given a compact group $A$, an algebraic action $G \cc A$ is a homomorphism from $G$ into the group of continuous group isomorphisms of $A$. Thus, two algebraic actions are conjugate if they are intertwined by a continuous group isomorphism.
\begin{claim}
\label{clm:complete-invariant}
Two orthogonal representations of $G \cc^\pi \cH$ and $G \cc^{\pi'} \cH'$ are additive conjugates if and only if $G \cc \bH$ and $G \cc \bH'$ are conjugate algebraic actions. 
\end{claim}
\begin{proof}
Assume first that $\xi\colon \cH \to \cH'$ witnesses the additive conjugacy of $\pi$ and $\pi'$. Then $\cH$ and $\cH'$ are isomorphic as {\em discrete} abelian groups, and hence their algebraic duals $\bH$ and $\bH'$ are isomorphic as topological groups; this is witnessed by $\xi_* \colon \bH \to \bH'$, defined by 
\begin{align*}
    [\xi_*\varphi](v') = \varphi(\xi^{-1}v').
\end{align*}
It is straightforward to check that since $\xi$ intertwines $\pi$ and $\pi'$, it holds that $\xi_*g\varphi = g\xi_*\varphi$, and hence the actions $G \cc \bH$ and $G \cc \bH'$ are conjugate.

Conversely, assume that $\xi_* \colon \bH \to \bH'$ witnesses the conjugacy of $G \cc \bH$ and $G \cc \bH'$. Then in particular $\bH$ and $\bH'$ are isomorphic as topological groups, and hence their algebraic duals $\cH$ and $\cH'$---endowed with the discrete topology---are conjugate (see the end of \S\ref{sec:duality}). This is witnessed by $\xi \colon \cH \to \cH'$, defined by
\begin{align*}
    [\xi v](\varphi') = v(\xi_*^{-1}\varphi).
\end{align*}
It is again straightforward to check that since $\xi_*$ intertwines the actions on $\bH$ and $\bH'$, $\xi$ intertwines $\pi$ and $\pi'$. Likewise, $\xi$ is immediately seen to be additive. 
\end{proof}

Let $m$ be the unique Haar probability measure on the compact group $\bH$. Since $G$ acts on $\bH$ by automorphisms, it preserves $m$, and so $G \cc (\bH,m)$ is a probability measure preserving action.
\begin{claim}
The following are equivalent:
\begin{enumerate}
    \item The action $G \cc (\bH,m)$ is ergodic.
    \item The action $G \cc \bH$ is topologically transitive.
    \item The orbit $\{g v \,\vert\,g \in G\}$ is infinite for every nonzero $v \in \cH$.
\end{enumerate}
\end{claim}
\begin{proof}
  The first two conditions are equivalent by~\cite[Theorem 1.1]{schmidt2012dynamical} (in fact, this equivalence holds more generally for actions by automorphisms on compact groups). By a result of Halmos~\cite{halmos1943automorphisms} for $\Z$ actions, which was extended by Kaplansky to finitely generated groups~\cite{kaplansky1949groups},\footnote{See also \cite[Lemma 1.2 and remark (3) on page 9]{schmidt2012dynamical}.} non-ergodicity of the action $G \cc (\bH,m)$ is equivalent to the existence of a nonzero $\chi \in \widehat\bH$ with a finite $G$-orbit. It follows from Pontryagin duality that the dual $\widehat\bH$ of $\bH$ can be identified with $\cH$, equipped with the discrete topology (see the end of \S\ref{sec:duality}). Thus a nonzero character $\chi \in \widehat\bH$ with a finite orbit is simply equivalent to a nonzero vector $v \in \cH$ with a finite orbit.
\end{proof}

\section{Proofs}

\subsection{Preliminary claims.}

\begin{claim}
\label{clm:harmonic}
Every fixed point of $\bH$ is $\mu$-harmonic for every symmetric generating measure $\mu$.
\end{claim}
\begin{proof}
Suppose $\mu$ is a symmetric generating measure with support $S$ and $\varphi$ is a fixed point of $\bH$. Then for any $v \in H$,
\begin{align*}
    \varphi(D_{\mu} v) = \varphi\left(v - \sum_{g \in S}{\mu(g) gv}\right)
    = \varphi(v) - \sum_{g \in S}{\varphi(g (\mu(g) v))}.
\end{align*}
Because $\varphi$ is a fixed point, $\varphi(g (\mu(g) v))=\varphi(\mu(g)v)$. Hence,
\begin{align*}
    \varphi(D_\mu v) = \varphi(v) - \sum_{g \in S}{\varphi(\mu(g) v)} = \varphi(v) - \varphi\left(\sum_{g \in S}{\mu(g) v}\right).
\end{align*}
But $\mu$ is a probability measure, and so $\sum \mu(g)v =v$. We thus obtain 
\begin{align*}
    \varphi(D_\mu v) = \varphi(v) - \varphi(v) = 0.
\end{align*}
\end{proof}

\subsection{Proof of main theorem}

\text{}

\vs

\begin{lemma}
  \label{lem:almost-invariant-orthogonal}
  Suppose $\pi$ does not have a nonzero invariant vector. Then if it has almost invariant
  vectors, it has almost invariant vectors that are mutually
  orthogonal.
\end{lemma}
Note that the assumptions of the lemma imply that the Hilbert space is infinite dimensional. In finite dimensional spaces, whenever there are almost invariant vectors there are non-zero invariant vectors, since the unit sphere is compact.
\begin{proof}[Proof of Lemma~\ref{lem:almost-invariant-orthogonal}]
  
Let $(v_n)_n$ be a sequence of almost invariant vectors. Since the unit ball is weakly sequentially compact, there is some subsequence $(v_{n_k})_k$ so that $v_{n_k} \weakly v$ for some $\|v\| \le 1$. For ease, relabel $(v_{n_k})_k = (v_n)_n$. We show $v$ is $G$-invariant: for any $g \in G$, we have $gv_n \weakly gv$ and so $gv_n-v_n \weakly gv-v$. Since $gv_n-v_n \to 0$, we must have $gv-v = 0$. By assumption, $v$ must be equal to $0$. Thus $(v_n)_n$ is a sequence of almost invariant (unit) vectors that weakly converge to $0$. 

\vs

We construct a slightly altered sequence $(w_k)_k$ of mutually orthogonal vectors that are still almost-invariant. Let $w_1 = v_1$. With $w_1,\ldots,w_k$ chosen, choose $m$ large enough so that  $m>k$ and
$$
  |\lan v_{m},w_{i} \ran| < \frac{1}{k}
$$ 
for $1 \le i \le k$, which is possible since, for each $1 \le i \le k$, $\lan w_{i}, v_m \ran \to 0$ as $m \to \infty$. Let $\hat v_{m}$ denote the projection of $v_{m}$ onto the linear space spanned by $w_{1},\dots,w_{k}$, and note that
$$
  \|\hat v_{m}\|^2 = \sum_{i=1}^k |\lan v_{m},w_{i} \ran|^2 < \frac{1}{k}.
$$  
Let
$$
  w_{k+1} = \frac{v_{m}-\hat v_{m}}{\|v_{m}-\hat v_{m}\|}.
$$ 
By construction the $w_k$'s are mutually orthogonal, each with norm $1$. Observe that for any $g \in G$,
\begin{align*}
\|gw_{k+1}-w_{k+1}\| 
&= \frac{1}{\|v_{m}-\hat v_{m}\|}\|gv_{m}-v_{m}+\hat v_{m}-g \hat v_{m}\| \\
&\le \frac{1}{1-1/\sqrt{k}}\left[\|gv_{m}-v_{m}\|+2\|\hat v_{m}\|\right] \\
&\le \frac{1}{1-1/\sqrt{k}}\left[\|gv_{m}-v_{m}\|+\frac{2}{\sqrt{k}}\right].
\end{align*}
Thus, since $m > k$, $(w_k)_k$ is an almost invariant sequence.
\end{proof}

\begin{proposition}
\label{prop:has-almost-invariant}
If $\pi$ has almost invariant vectors, then $\bH$ has a nonzero fixed point. 
\end{proposition}

\begin{proof}
If $\pi$ has a nonzero invariant vector, say $w$, then $\sigma_w$ is clearly a nonzero fixed point of $\bH$. So we may suppose $\pi$ has no nonzero invariant vector. Then, by Lemma \ref{lem:almost-invariant-orthogonal}, we may let $(v_n)_n$ be a sequence of almost invariant (unit) vectors such that for all $m \not = n$, $\langle v_m,v_n \rangle = 0$.

\vs

Fix a finite generating set $S$ for $G$. By passing to a subsequence of $(v_n)_n$ if necessary, we may suppose that for all $n \ge 1$ and all $h \in S$, we have \[ \Vert hv_n-v_n\Vert < \frac{1}{2^n}.\] For each $n \ge 1$, let 
$$
\ep_n = \max_{h \in S} \Vert hv_n-v_n\Vert,
$$
and define $$w_n = \frac{v_n}{\sqrt{\ep_n}}.$$

\vs

Since $\bH$ is compact, there is a subnet $(w_\alpha)_\alpha$ of $(w_n)_n$ so that $(\sigma_{w_\alpha})_\alpha$ converges, say to $\varphi \in \bH$. 

\vs

We first show $\varphi \not = 0$. To this end, let $$w = \frac{1}{2}\sum_{n=1}^\infty \epsilon_n w_n,$$ which is well-defined since $$\sum_{n=1}^\infty |\epsilon_n w_n| = \sum_{n=1}^\infty \sqrt{\epsilon_n} < \sum_{n=1}^\infty \sqrt{2^{-n}} < \infty.$$ Note, for any $n \ge 1$, $$\langle w,w_n \rangle = \frac{1}{2}\epsilon_n\langle w_n,w_n \rangle = \frac{1}{2}$$ by orthogonality. Therefore, $$\varphi(w) = \lim_\alpha \langle w_\alpha,w\rangle+\Z = \frac{1}{2}+\Z$$ shows indeed $\varphi \not = 0$. 

\vs

We finish by showing $\varphi$ is a fixed point of $\bH$. Take $v \in \mc{H}$ and $h \in S$. Then 
$$
  \varphi(h^{-1}v)-\varphi(v) = \lim_\al (\sigma_{w_\al}(h^{-1}v)-\sigma_{w_\al}(v)) = \lim_\al (\langle v,hw_\al\rangle - \langle v,w_\al \rangle+\Z).
$$ 
Now, 
\begin{align*} 
\lim_\al |\langle v,hw_\al\rangle - \langle v,w_\al \rangle|  &= \lim_\al |\langle v,hw_\al-w_\al \rangle| \\ 
  &\le \lim_\al \Vert v \Vert \Vert hw_\al-w_\al \Vert.
\end{align*}
Since $$\Vert hw_n-w_n \Vert \le \sqrt{\ep_n} \to 0,$$ it holds that $$\Vert hw_\al-w_\al \Vert \to 0,$$ and therefore we see $\varphi(h^{-1}v) = \varphi(v)$ for each $h \in S$. Since $S$ is a symmetric generating set, $\varphi$ is a fixed point of $\bH$.
\end{proof}

The folloming lemma is standard; we omit the proof. We will be use it to show $D_\mu$ is surjective.
\begin{lemma} 
\label{lem:closed}
Let $T: \cH \to \cH$ be a bounded linear operator. Suppose there exists some $c > 0$ so that $\Vert Tx\Vert \ge c\Vert x \Vert$ for all $x \in H$. Then the range of $T$ is closed.
\end{lemma}
%\begin{proof}
%If $y_n \to y$ with $Tx_n = y_n$, then 
%$$c\Vert x_n-x_m\Vert \le \Vert Tx_n-Tx_m \Vert = \Vert y_n-y_m\Vert \to 0,$$ 
%so $(x_n)_n$ is Cauchy. Say $x_n \to x$; then, since $T$ is bounded, $Tx = \lim_n Tx_n = \lim_n y_n = y$, as desired.
%\end{proof}

\begin{proposition}
\label{prop:no-almost-invariant}
If $\pi$ does not have almost invariant vectors, then $D_\mu$ is surjective for any symmetric generating measure $\mu$.
\end{proposition}

\begin{proof}
First observe that $D_\mu$ is self-adjoint since $\mu$ is symmetric. Furthermore, $D_\mu$ is injective since there are no nonzero invariant vectors. Indeed, note that if $D_\mu(w) = 0$ for some unit vector $w$, then 
$$
  w = \sum_h \mu(h) hw
$$
and so
$$
  1 = \sum_h \mu(h) \lan hw, w \ran.
$$
The right hand side is the average of numbers that are at most $1$. Since this average is equal to $1$ they all have to equal $1$, and so (since $\mu$ is generating) $w$ is invariant.

Since $D_\mu$ is self-adjoint and injective, $D_\mu$ has a dense image. Hence, by Lemma~\ref{lem:closed}, it suffices to show the lower bound inequality. Since, for $\Vert v \Vert = 1$, 
$$
  \Vert D_\mu(v)\Vert  = \Vert v-P_\mu(v)\Vert  \ge \Vert v\Vert -\Vert P_\mu(v)\Vert  = 1-\Vert P_\mu(v)\Vert ,
$$ 
it suffices to bound $\Vert P_\mu(v)\Vert $ away from $1$. This is an immediate consequence of \cite[Proposition 6.2.1]{bekka2008kazhdan}, but we include a proof for completeness.

Let $S$ denote the support of $\mu$. Since there are no almost invariant vectors, there is an $\ep > 0$ so that for all  $\Vert v\Vert  = 1$ there exists an $h \in S$ such that $ \Vert hv-v\Vert  \ge \ep$ (see, e.g., \cite[Proposition F.1.7]{bekka2008kazhdan}). Note for such a $v$ and $h$, we have $$\ep^2 \le \Vert hv-v\Vert ^2 = \lan hv-v, hv-v \ran = 1-2\lan v,hv \ran +1$$ and thus $$\lan v,hv \ran \le 1-\frac{1}{2}\ep^2.$$ 

Therefore, for any unit vector $v \in \cH$, taking again $h_0 \in S$ so that $\Vert h_0v-v\Vert  \ge \ep$ gives 

\begin{align*}
\Vert P_\mu v\Vert ^2 
&= \left \lan \sum_{h \in S} \mu(h) hv, \sum_{h' \in S} \mu(h')h'v\right \ran\\
&=\mu(h_0)\mu(e)\lan h_0v,v\ran + \sum_{(h,h') \not = (h_0,e)} \mu(h)\mu(h')\lan hv,h'v \ran \\
&\le \mu(h_0)\mu(e)(1-\frac{1}{2}\ep^2)+\sum_{(h,h') \not = (h_0,e)} \mu(h)\mu(h')\\
&= 1-\frac{1}{2}\mu(h_0)\mu(e)\ep^2.
\end{align*}

Consequently, for each $\Vert v \Vert = 1$, $$\Vert P_\mu v \Vert^2 \le 1-\frac{1}{2}\ep^2 \mu(e)\inf_{h \in S} \mu(h).$$ Since $S$ is finite, we are done.
\end{proof}

We now have the tools to prove Proposition \ref{prop:harmonic}.

\begin{proof}[Proof of Proposition \ref{prop:harmonic}]
Suppose $\pi$ does not have almost invariant vectors, $\mu$ is a symmetric generating measure, and $\varphi$ is a $\mu$-harmonic point of $b\cH$. Then $\varphi(D_\mu(\cH)) = 0$. Now, by Proposition \ref{prop:no-almost-invariant}, since $\pi$ does not have almost invariant vectors, $D_\mu$ is surjective. Hence, $\varphi(\cH) = 0$, and so $\varphi = 0$. Thus, the only $\mu$-harmonic point of $b\cH$ is $0$. Further, since every fixed point of $b\cH$ is $\mu$-harmonic (Claim~\ref{clm:harmonic}), it follows that the only fixed point of $b\cH$ is $0$. 
\end{proof}

We are now ready to prove our main theorem.
\begin{proof}[Proof of Theorem~\ref{thm:fixed-point}]
First, suppose that $\pi$ has almost invariant vectors. Then by Proposition~\ref{prop:has-almost-invariant}, $b\cH$ has a nonzero fixed point. Now, suppose that $b\cH$ has a nonzero fixed point. Then by Proposition~\ref{prop:harmonic}, $\pi$ must have almost invariant vectors.
\end{proof}

\subsection{The Bohr compactification of a separable Hilbert space}

Let $\cH$ be a separable real Hilbert space. Recall that the {\em algebraic dual} of $\cH$ is given by
\begin{align*}
    \widehat{\cH} = \{\phi: \cH \to \T \hs | \hs \phi \text{ is a continuous  homomorphism}\}
\end{align*}
As above, for each $v \in \cH$, let $\sigma_v \colon \cH \to \T$ be given by
$$
\sigma_v(w) = \lan v,w \ran + \Z.
$$

\begin{proposition}
\label{prop:H-dual}
A real separable Hilbert space $\cH$, equipped with the weak topology, can be identified as a topological group with its algebraic dual $\widehat{\cH}$ via $v \mapsto \sigma_v$.
\end{proposition}

\begin{proof}[Proof of Proposition \ref{prop:H-dual}]
We first note that every $\sigma_v$ is an element of $\widehat{\cH}$, i.e., is a continuous homomorphism from $\cH$ to $\T$. This follows from the fact that $w \mapsto \lan v,w\ran$ is weakly continuous, and that the projection $\R \to \T$ is also continuous. We now  show that every continuous homomorphism $\phi : \cH \to \T$ is equal to some $\sigma_w$. 

Given $\phi$, let $\psi \colon \cH \to \R$ be the unique lift of $\phi$ satisfying $\psi(0)=0$. Then $\psi$ is continuous and is easily seen to furthermore be linear. Hence it must be of the form $\psi(v) = \lan v,w \rangle$ for some $w \in \cH$. Since $\psi$ is a lift of $\phi$, $\phi(v) = \langle v,w \rangle + \Z$. We have thus shown that $\phi = \sigma_w$.

Finally, we argue that $v \mapsto \sigma_v$ is continuous and has a continuous inverse. It follows immediately from the definition that if a net $(v_\alpha)_\alpha$ converges weakly to $v$ then $(\sigma_{v_\alpha})_\alpha$ converges to $\sigma_v$. Hence $v \mapsto \sigma_v$ is continuous. Since the unit ball in $\cH$ is weakly compact, the restriction of this map to this ball has an inverse that is also continuous. By the additivity of the map $v \mapsto \sigma_v$, it follows that this map has a continuous inverse.
\end{proof}

\subsection{Proof of Proposition~\ref{prop:weak-containment}}

\begin{lemma}
\label{lem:weak-to-bohr}
Let $\rho: G \to \T$ be a function. Suppose there exist vectors $(v_n)_n$ that weakly converge to $0$, and that for each $g \in G$, $\lim_n \langle gv_n, v_n \rangle+\Z = \rho(g)$. Then there exist $\varphi \in \bH$ and $v \in \cH$ so that $\varphi(g^{-1}v) = \rho(g)$ for each $g \in G$.
\end{lemma}

\begin{proof}
Let $\{g_1,g_2,\dots\}$ be an enumeration of (the countable group) $G$. Let $w_1 = v_1$. With $w_1,\dots,w_{n-1}$ chosen, let $w_n = v_{k_n}$ be such that $|\langle v_{k_n}, gw_j\rangle| < 2^{-n^2}$ for each $1 \le j \le n-1$ and $g \in \{g_1,g_1^{-1},\dots,g_{n-1},g_{n-1}^{-1}\}$, which is possible since $v_n \weakly 0$.  Let $v = \sum_n 2^{-n}w_n$. As $\bH$ is compact, we may take $\varphi$, a limit of some subnet of $(2^n\sigma_{w_n})_n$. Fix $g \in G$. Note, for each $n \ge 1$, $$2^n\sigma_{w_n}(g^{-1}v) = \langle gw_n,w_n\rangle + \sum_{k < n} 2^n2^{-k} \langle g^{-1}w_k,w_n \rangle + \sum_{k > n} 2^n2^{-k} \langle w_k, gw_n \rangle.$$
By construction, if $g = g_l$ and $n > l$, then for
$k < n$, $|\langle g^{-1}w_k,w_n \rangle| \le 2^{-n^2}$ and for $k > n$, $|\langle w_k, gw_n \rangle| \le 2^{-k^2}$. So, we obtain that 
\begin{align*}
    \lim_n 2^n \sigma_{w_n}(g^{-1}v) &= \lim_n{\left(\langle gw_n,w_n\rangle +O\left(n2^n2^{-n^2}+\sum_{k > n} 2^{n-k^2}\right)\right)}\\
    &= \rho(g).
\end{align*}
We conclude $\varphi(g^{-1}v) = \rho(g)$.
\end{proof}

\begin{proof}[Proof of Proposition~\ref{prop:weak-containment}]
Since $\pi$ weakly contains $\pi'$ and $\pi'$ is irreducible, for any $v' \in \cH'$, there exists a sequence of vectors $(v_n)_n$ in $\cH$ such that $\lim{\lan g v_n, v_n \ran} = \lan gv', v' \ran$ for each $g \in G$ (see, e.g., \cite[Proposition F.1.4]{bekka2008kazhdan}). Since the norm of $v_n$ converges to the norm of $v'$, we may assume that $(v_n)_n$ has a weak limit, say $v$. If $v=0$, then the result follows immediately from Lemma \ref{lem:weak-to-bohr}. So we may assume $v \neq 0$.

For each $n$, let $u_n = v_n - v$. For any $g \in G$, it follows from $u_n \weakly 0$ that
\begin{align*}
    \lan gv',v'\ran 
    &=\lim_n{\lan g v_n, v_n \ran} \\
    &= \lim_n{\lan g (u_n + v), (u_n + v) \ran}\\
    &= \lim_n{\lan g u_n, u_n \ran} + \lim_n{\lan g u_n, v \ran} + \lim_n{\lan g v, u_n \ran} + \lim_n{\lan g v, v \ran}\\
    &= \lim_n{\lan g u_n, u_n \ran} + \lan g v, v \ran.
\end{align*}

Since $\pi'$ is irreducible, positive functions associated to it are extreme points in the  cone of positive functions, a cone which is closed with respect to pointwise convergence (see, e.g., \cite[Proposition C.5.2]{bekka2008kazhdan}). Therefore, $\lim_n{\lan g u_n, u_n \ran} = t \lan gv', v' \ran$ for some $t \in \R$. In particular, $\lan gv',v'\ran = \frac{1}{1-t} \lan g v, v \ran$ ($t \neq 1$ since $v \neq 0$). The result follows by taking $\varphi = \sigma_{v/(t-1)}$.
\end{proof}

\section{Additive Conjugacy of Irreducible Representations of $\Z$}
\label{sec:Z-reps}

We would like to thank Yehuda Shalom for suggesting to us the results of this section.

The complex irreducible unitary representations of $\Z$ are one-dimensional, and given by multiplication by a complex number $z \in \C$ with $|z| = 1$. We denote by $\pi^z$ the representation associated to $z$, so that for $n \in \Z$ and $x \in \C$ we have $\pi^z_n x = z^n x$. 

\begin{theorem}
Two irreducible $\Z$ representations $\pi^z$ and $\pi^w$ are additive conjugates if and only if for every $p(x) \in \Z[x]$ it holds that $p(z) = 0$ if and only if $p(w) = 0$.
\end{theorem}

\begin{proof}
First suppose $\pi^z,\pi^w$ are additive conjugates, witnessed by an additive bijection $\xi$. Take $p(y) \in \Z[y]$ with $p(z) = 0$. Write $p(y) \equiv r_ny^n+\dots+r_1y+r_0$. Take $x \in \C$ with $\xi(x) \not = 0$. Then 
\begin{align*}
0 &= \xi(r_nz^nx+\dots+r_1zx+r_0x)\\
&= r_n(w)^n\xi(x)+\dots+r_1 w \xi(x)+r_0\xi(x) \\
&= p(w)\xi(x),
\end{align*}
where the first equality is a consequence of $p(z)=0$, and the second follows from the additivity of $\xi$ and the fact that it intertwines $\pi^z$ and $\pi^w$.

This gives $p(w) = 0$. We have shown every $p(y) \in \Z[y]$ with $p(z) = 0$ has $p(w) = 0$. By symmetry, of course, we get the reverse.

\vs

Now suppose that for every $p(y) \in \Z[y]$, it holds that $p(z) = 0$ if and only if $p(w) = 0$.
It is then well known that there is an isomorphism\footnote{One can take  $\Xi(\frac{p(z)}{q(z)}) = \frac{p(w)}{q(w)}$. It is easy to verify that this is a well-defined isomorphism.} $\Xi: \Q(z) \to \Q(w)$ with $\Xi(z) = w$. 
%To see this, define $\Xi: \Q(z) \to \Q(w)$ via $$\Xi\left(\frac{p(z)}{q(z)}\right) = \frac{p(w)}{q(w)}$$ for any $p(y),q(y) \in \Q[y]$ with $q(z) \not = 0$ (note that, by assumption, $q(w) \not = 0$). Also, if $\frac{p_1(z)}{q_1(z)} = \frac{p_2(z)}{q_2(z)}$, then $$p_1(z)q_2(z)-p_2(z)q_1(z) = 0,$$ 
%which implies 
%$$p_1(w)q_2(w)-p_2(w)q_1(w) = 0$$ 
%since $p_1q_2-p_2q_1 \in \Q[y]$. This implies $\frac{p_1(w)}{q_1(w)} = \frac{p_2(w)}{q_2(w)}$, or in other words, that $\Xi(\frac{p_1(z)}{q_1(z)}) = \Xi(\frac{p_2(z)}{q_2(z)})$. We have thus shown $\Xi$ is well-defined, and it is clear that $\Xi$ is then an additive bijection. Finally, for any $p(y),q(y) \in \Q[y]$ with $q(z) \not = 0$, $$\Xi\left(z\frac{p(z)}{q(z)}\right) = \Xi\left(\frac{zp(z)}{q(z)}\right) = \frac{wp(w)}{q(w)} = w\Xi\left(\frac{p(z)}{q(z)}\right).$$

Let $\{x_\al\}_\al$ be a basis for $\C$ over $\Q(z)$ and $\{y_\al\}_\al$ be a basis for $\C$ over $\Q(w)$. Define $\xi: \C \to \C$ by $$\xi\left(\sum_\al c_\al x_\al\right) = \Xi(c_\al) y_\al,$$ 
which is well-defined (on all of $\C$) since $\{x_\al\}$ is a basis. Also, that $\Xi$ is a field isomorphism mapping $z$ to $w$ implies that $\xi$ is an additive bijection intertwining $\pi^z$ and $\pi^w$:
\begin{align*}
\xi\left(z^n\sum_\al c_\al x_\al\right) =  
\sum_\al \Xi (z^nc_\al)y_\al 
= \sum_\al w^n\Xi(c_\al)y_\al 
= w^n\xi\left(\sum_\al c_\al x_\al\right).
\end{align*}
\end{proof}

\begin{corollary}
The representations $\pi^z$ and $\pi^w$ are additive conjugates if and only if both $z,w$ are not algebraic numbers or they are both algebraic numbers and are Galois conjugates. 
\end{corollary}
Note that it is easy to show that if, for example, $z$ and $w$ are not algebraic then  $\pi^z$ and $\pi^w$ cannot be conjugate representations (in the usual sense) unless $z=w$. Thus this is an example of irreducible representations that are additive conjugates but not conjugates.

\bibliography{refs}

\end{document}